\documentclass[11pt]{amsart}
\usepackage{amssymb,latexsym,amsmath,graphicx,graphics,epic,eepic}

\theoremstyle{plain}
\newtheorem{theorem}{Theorem}
\numberwithin{theorem}{section}
\newtheorem{lemma}[theorem]{Lemma}
\newtheorem{proposition}[theorem]{Proposition}
\newtheorem{corollary}[theorem]{Corollary}
\newtheorem{conjecture}[theorem]{Conjecture}

\theoremstyle{definition}

\newcommand{\C}{{\mathbb C}}
\newcommand{\R}{{\mathbb R}}
\newcommand{\Z}{{\mathbb Z}}
\newcommand{\Q}{{\mathbb Q}}
\renewcommand{\P}{{\mathbb P}}
\newcommand{\s}{{\mathbb S}}

\newcommand{\I}{{\mathbb I}}

              \renewcommand{\L}{{\mathcal L}}

\begin{document}
\title{Seifert fibered four-manifolds with nonzero Seiberg-Witten invariant}
\author{Weimin Chen}
\subjclass[2000]{}
\thanks{The author is partially supported by NSF grant DMS-1065784.}
\date{\today}
\maketitle

\begin{abstract}
The main result of this paper asserts that if a Seifert fibered $4$-manifold has nonzero
Seiberg-Witten invariant, the homotopy class of regular fibers has infinite order. This is a
nontrivial obstruction to smooth circle actions; as applications, we show how to destroy 
smooth circle actions on a $4$-manifold by knot surgery, without changing the integral 
homology, intersection form, and even the Seiberg-Witten invariant. Results concerning 
classification of Seifert fibered complex surfaces or symplectic $4$-manifolds are included. 
We also show that every smooth circle action on the $4$-torus is smoothly conjugate to a 
linear action.
\end{abstract}

\section{Introduction}

A $4$-manifold is called {\it Seifert fibered} if it admits a smooth, fixed-point free circle action.
As such, the $4$-manifold is given as the total space of a circle bundle over a $3$-dimensional orbifold, whose singular set consists of at most disjoint embedded circles. An oriented Seifert fibered $4$-manifold must have zero Euler number and signature, which seems to be the only known obstructions. For basic results concerning smooth circle actions on 
$4$-manifolds, we refer the reader to \cite{F1,F2}.

In this paper we derive new obstructions for Seifert fibrations on $4$-manifolds which have 
nontrivial Seiberg-Witten invariant. Recall that for an oriented $4$-manifold $M$ with 
$b^{+}_2>1$, the 
Seiberg-Witten invariant of $M$ is a map $SW_M$ that assigns to each $Spin^c$-structure $\L$ 
of $M$ an integer $SW_M(\L)$ which depends only on the diffeomorphism class of $M$ (cf. \cite{M}).  In the case of $b_2^{+}=1$, the definition of $SW_M(\L)$ requires an additional choice of orientation of the $1$-dimensional space $H^{2,+}(M,\R)$, as discussed in Taubes \cite{T2}. A convention throughout this paper is that when we say $M$ has nonzero Seiberg-Witten invariant in the case of $b_2^{+}=1$, it is meant that the map $SW_M$ has nonzero image in $\Z$ for {\it some} choice of orientation of $H^{2,+}(M,\R)$. Note that under this convention, every symplectic $4$-manifold has
nonzero Seiberg-Witten invariant by the work of Taubes \cite{T1}.

Baldridge in \cite{Bald2} showed how to compute the Seiberg-Witten invariant of a Seifert fibered 
$4$-manifold in terms of the base $3$-orbifold. We observe that, through a de-singularization 
formula of Seiberg-Witten invariants of $3$-orbifolds whose details are given in a separate 
paper \cite{C1}, the non-vanishing of Seiberg-Witten invariant yields topologically interesting
informations about the base $3$-orbifold. For the purpose here, we paraphrase the 
said de-singularization formula as follows: 

\vspace{2mm}
{\it Let $Y$ be a closed, oriented $3$-orbifold with $b_1>1$, and let $Y_0$ be the oriented 
$3$-orbifold obtained from $Y$ by removing an embedded circle from its singular set. Then
$Y$ has nonzero Seiberg-Witten invariant if and only if $Y_0$ does.
}
\vspace{2mm}

The center of a group $G$ will be denoted by $z(G)$.

\subsection{Topological constraints}

Among the new constraints of Seifert fibered $4$-manifolds discussed in this paper, the central 
result, as stated in the following theorem, asserts that the fundamental group of the $4$-manifold 
has infinite center. The class of Seifert fibered $4$-manifolds whose $\pi_1$ has infinite center is
further studied in \cite{C2}, where the main tools are from $3$-dimensional topology centered
around the recently proved Thurston's Geometrization Conjecture (cf. \cite{BLP, P}).

\begin{theorem}
Let $X$ be an oriented $4$-manifold with $b_2^{+}\geq 1$ and nonzero Seiberg-Witten invariant, 
and let $\pi:X\rightarrow Y$ be any Seifert fibration. Then (1) the homotopy class of a regular fiber 
of $\pi$, which lies in $z(\pi_1(X))$, has infinite order, and (2) if $b_2^{+}>1$, the Hurwitz
homomorphism $\pi_2(X)\rightarrow H_2(X)$ has finite image.
\end{theorem}

In the course of the proof, we will show that $Y$ has a regular, finite manifold cover,
$Y=\tilde{Y}/G$. The pull-back of $\pi$ to $\tilde{Y}$, denoted by $\tilde{X}$, is a circle bundle 
over a $3$-manifold, which finitely covers $X$. 

It can be shown that the $4$-manifold $X$ in Theorem 1.1 is minimal, and moreover, 
$z(\pi_1(X))$ is finitely generated and torsion-free (details are given in \cite{C2}).

\vspace{3mm}

\noindent{\bf Remarks:}\hspace{1mm}
(1) A special case of Theorem 1.1(1), where $X$ is symplectic and the circle action is free, 
was due to Kotschick \cite{Ko1} (see also \cite{Bow}) using a different argument.  As already 
hinted in \cite{Ko1}, the assumption of non-vanishing Seiberg-Witten invariant is absolutely necessary. For instance, 
take $X$ to be the fiber-sum of $\s^1\times N$ and $\s^1\times \s^3$ where $N$ is a closed 
oriented $3$-manifold with $b_1(N)>1$. Then $X$ is a circle bundle over $N\# \s^1\times \s^2$, 
and $\pi_1(X)=\pi_1(N\# \s^1\times \s^2)=\pi_1(N) \ast \Z$ which is centerless.  Since 
$N\# \s^1\times \s^2$ has zero Seiberg-Witten invariant and $b_2^{+}(X)=b_1(N)>1$, the 
Seiberg-Witten invariant of $X$ is also zero (cf. \cite{Bald1}).

(2) There is a $4$-manifold which satisfies all the known constraints yet is not Seifert fibered: 
Let $X$ be the projective bundle $\P(E\times \C)$ where $E$ is a holomorphic line bundle of 
odd degree over an elliptic curve. Then $X$ is minimal with nonzero Seiberg-Witten invariant, 
and $z(\pi_1(X))=\pi_1(X)=\Z^2$. The claim that $X$ is not Seifert fibered follows from a 
classification theorem in \cite{C2}, which says that a Seifert fibered $4$-manifold with 
$z(\pi_1)=\Z^2$ and $\pi_2\neq 0$ must be diffeomorphic
to $T^2\times \s^2$. (Note that although $X$ is not Seifert fibered, a double cover of $X$, which is diffeomorphic to $T^2\times \s^2$, is Seifert fibered.)

\vspace{2mm}

As a corollary of Theorem 1.1 and the aforementioned classification result in \cite{C2}, the class of complex surfaces which may possess a Seifert fibration is determined. Except for the class $VII$ surfaces, a complex surface in all other cases clearly admits a Seifert fibration.

\begin{theorem}
A complex surface is not Seifert fibered unless it is either a class $VII$ 
surface with $b_2=0$, or has an elliptic fibration with trivial monodromy 
representation, or is a ruled surface diffeomorphic to $T^2\times \s^2$. 
\end{theorem}

Examples of class $VII$ surfaces with $b_2=0$ are Hopf surfaces or Inoue surfaces. Many Hopf
surfaces admit a smooth fixed-point free circle action, hence are naturally Seifert fibered. However,
there are no naturally defined circle actions on an Inoue surface (cf. \cite{BPV}).  Note that since 
these $4$-manifolds have $b_2=0$, Seiberg-Witten invariants are not well-defined. Consequently, 
for a given class $VII$ surface with $b_2=0$ which is not naturally Seifert fibered, the question as whether it is indeed not Seifert fibered is currently beyond reach. 

\vspace{2mm}

In the following theorem, we give a classification of smooth circle actions on $T^4$.

\begin{theorem}
Every smooth circle action on the $4$-torus is smoothly conjugate to a linear action. 
\end{theorem}

\subsection{Destruction of circle actions}

Theorem 1.1 gives new obstructions to smooth circle actions
on $4$-manifolds. For convenience we shall collect in the next theorem all the known obstructions, with the last item being a corollary of Theorem 1.1 and a theorem of Baldridge (cf.  \cite{Bald3}, Theorem 1.1). The latter says that a $4$-manifold with $b_2^{+}>1$ and admitting a smooth circle action with nonempty fixed-point set must have vanishing Seiberg-Witten invariant. 

\begin{theorem}
{\em(}Obstructions for smooth circle actions on $4$-manifolds{\em)}
\begin{itemize}
\item [{(1)}] (Atiyah-Hirzebruch \cite{AH}, Herrera \cite{H}) If a $4$-manifold with even 
intersection form admits a smooth circle action, then its signature must vanish.
\item [{(2)}] (Fintushel \cite{F2}\footnote{This was proved modulo the $3$-dimensional 
Poincar\'{e} Conjecture, which is now resolved, cf. \cite{P}.}) If a simply connected $4$-manifold admits a smooth circle action, then it must be diffeomorphic to a connected sum of $\s^4$, 
$\pm \C\P^2$, or $\s^2\times \s^2$. 
\item [{(3)}] (Baldridge \cite{Bald3}) Let $X$ be a $4$-manifold with either nonzero Euler 
characteristic or nonzero signature, and suppose $X$ admits a smooth circle action. (a) If $X$ is symplectic, then $X$ must be diffeomorphic to a rational or ruled surface. (b) If $b_2^{+}>1$, then 
$X$ must have vanishing Seiberg-Witten invariant.
\item [{(4)}] Let $X$ be a $4$-manifold with $b_2^{+}>1$ admitting a smooth circle action. If $X$ 
has nonzero Seiberg-Witten invariant, then {\em(i)} $z(\pi_1(X))$ must be infinite, and {\em(ii)} 
the Hurwitz homomorphism $\pi_2(X)\rightarrow H_2(X)$ must have finite image. 
\end{itemize}
\end{theorem}

We fell upon these new obstructions (i.e. Theorem 1.4(4)) 
when investigating methods for destroying smooth circle actions on $4$-manifolds 
using the Fintushel-Stern knot surgery \cite{FinS}, as motivated
by considerations in \cite{C}. In order to explain this, suppose we are given
a Seifert fibered $4$-manifold $X$ with $b_2^{+}>1$. 
We denote by $\pi:X\rightarrow Y$ the Seifert fibration, and furthermore, 
we assume that there is an embedded loop $l\subset X$ satisfying the following conditions:
\begin{itemize}
\item [{(i)}] $\pi(l)$ is an embedded loop lying in the complement of the singular set of $Y$, 
and $l$ is a section of $\pi$ over $\pi(l)$, 
\item [{(ii)}] no nontrivial powers of the homotopy class of $l$ are contained in $z(\pi_1(X))$,
\item [{(iii)}] the $2$-torus $T\equiv \pi^{-1}(\pi(l))$ is non-torsion in $H_2(X)$. 
\end{itemize}

Now for any (nontrivial) knot $K\subset\s^3$, we let $X_K$ be the $4$-manifold obtained
from knot surgery on $X$ along $T$ using knot $K$, i.e., 
$$
X_K\equiv X\setminus Nd(T)\cup_\phi (\s^3\setminus Nd(K))\times \s^1.
$$ 
Here in the knot surgery we require that 
the diffeomorphism $\phi$ identifies the meridian of $T$ with the 
longitude of $K$, a fiber of $\pi$ with the meridian of $K$, and a push-off of 
$l$ with $\{pt\}\times \s^1$. It follows easily from the Mayer-Vietoris
sequence that $X_K$ has the same integral homology and the same intersection
pairing of $X$, and moreover, the Seiberg-Witten invariant of $X_K$ can be
calculated from that of $X$ and the Alexander polynomial of $K$ (cf. \cite{FinS, F3}).

With the preceding understood, we have the following theorem.

\begin{theorem}
Let $X$ be a Seifert fibered $4$-manifold with $b_2^{+}>1$ such that 
$$
SW_X(z) \equiv \sum_{c_1(\L)=z} SW_X(\L)\neq 0
$$
for some $z\in H^2(X)$. Let there be an embedded loop $l$ in $X$ satisfying $(i)-(iii)$. 
Suppose $K$ is not a torus knot. Then the $4$-manifold $X_K$ constructed above
does not support any smooth circle actions.
\end{theorem}

We remark that when $K$ has trivial Alexander polynomial and $H^2(X)$ has no $2$-torsions,  
$X_K$ and $X$ have the same Seiberg-Witten invariant.

Consider the following example. Let $X=S^1\times N^3$, where
$N^3= [0,1]\times T^2/\sim$ with $(0,x,y)\sim (1,x+y,y)$. Then $X$ is 
naturally a $T^2$-bundle over $T^2$, with $x,y$ being coordinates on the fiber.
Note further that the translations along the $x$-direction define naturally a free
smooth circle action on $X$, so in this sense $X$ is Seifert fibered. The integral 
homology of $X$ is given as follows:
$$
H_1(X)=H_3(X)=\Z\oplus\Z\oplus\Z, \mbox{ and } H_2(X)=\Z\oplus\Z\oplus\Z\oplus\Z.
$$
In particular, $b_2^{+}=2$. Finally, as a $T^2$-bundle over $T^2$, $X$ has a symplectic 
structure with $c_1(K_X)=0$.

In order to perform the construction of manifold $X_K$ described in Theorem 1.5,
we consider the following embedded loop $l$ in $X$ which clearly satisfies the
conditions (i)-(iii) required in the construction: we pick any $T^2$-fiber in $X$ and take in that 
fiber an embedded loop parametrized by the $y$-coordinate. Notice that in the present case 
the $2$-torus $T\equiv \pi^{-1}(\pi(l))$ is simply the $T^2$-fiber we picked. 

Now for any fixed, nontrivial knot $K\subset \s^3$ which is not a torus knot, we obtain 
a smooth $4$-manifold $X_K$ which has the same integral homology of $X$ and does 
not support any smooth circle actions. Furthermore, 
note that $X_K$ is symplectic if $K$ is chosen to be a fibered knot, and $X_K$ has the same 
Seiberg-Witten invariant of $X$ if the Alexander polynomial of $K$ is trivial. 

We shall also consider an equivariant version of the above construction. 
Note that the $\s^1$-factor in $X$ defines a natural circle action,
so that for any integer $p\geq 2$ there is an induced free smooth 
$\Z_p$-action on $X$, which preserves the $T^2$-bundle structure. 
For a fixed, nontrivial knot $K$, we let $X_{K,p}$ denote the $4$-manifold resulted from 
a repeated application of the above construction which is equivariant with respect 
to the free $\Z_p$-action. Using the fact that repeated knot surgery along parallel
copies is equivalent to a single knot surgery using the connected sum of the knot (cf. 
Example 1.3 in \cite{C}), we obtain the following corollary.

\begin{corollary}
Let $p\geq 2$ be an integer, and let $K\subset\s^3$ be a nontrivial knot. There are smooth $4$-manifolds $X_{K,p}$ such that
the following statements are true.
\begin{itemize}
\item [{(1)}] 
$H_1(X_{K,p})=H_3(X_{K,p})=\Z\oplus\Z\oplus\Z, \; H_2(X_{K,p})=\Z\oplus\Z\oplus\Z\oplus\Z$.
\item [{(2)}] $X_{K,p}$ does not admit any smooth circle actions, but has a smooth free $\Z_p$-action.
\end{itemize}
Moreover, 
\begin{itemize}
\item [{(3)}] If $K$ is a fibered knot, then $X_{K,p}$ has a symplectic structure $\omega$ where
$[\omega]\in H^2_{dR}(X_{K,p})$ is integral, with respect to which the free $\Z_p$-action on 
$X_{K,p}$ is symplectic and $c_1(K_{X_{K,p}})\cdot [\omega]$ grows 
linearly with respect to $p$. 
\item [{(4)}] If $K$ has trivial Alexander polynomial, then the Seiberg-Witten invariants of the
$4$-manifolds $X_{K,p}$ are all the same, i.e., independent of $p$ and $K$. 
\end{itemize}
\end{corollary}

\noindent{\bf Remarks:}\hspace{2mm}
In \cite{C} the author considered an extension of the classical Hurwitz theorem on automorphisms 
of surfaces to $4$ dimensions, where bounds for the order of periodic diffeomorphisms 
of $4$-manifolds admitting no smooth circle actions were investigated.
It was shown that the order of a holomorphic 
$\Z_p$-action is bounded by a constant depending only on the integral homology 
of the manifold, however, in the symplectic case the bound was shown to depend in addition on $c_1(K_X)\cdot [\omega]$, a term reflecting the smooth and symplectic structures of the 
$4$-manifold. Examples were given in \cite{C} which show that the involvement of 
$c_1(K_X)\cdot [\omega]$ is indeed necessary, at least in the case of $b_2^{+}=1$.  With this understood, note that Corollary 1.6(3) gives further examples for the case of $b_2^{+}>1$. Furthermore, Corollary 1.6(4) shows that the integral homology and the Seiberg-Witten invariants alone are insufficient in bounding the order of a smooth $\Z_p$-action on a $4$-manifold, and
that certain measurement for the complexity of fundamental group must also be involved. 

\subsection{Symplectic Seifert fibered $4$-manifolds}

Concerning smooth classification of Seifert fibered $4$-manifolds which admit a symplectic
structure, we have the following conjecture.

\begin{conjecture}
Let $X$ be a symplectic $4$-manifold and $\pi:X\rightarrow Y$ be a Seifert fibration. 
Then $Y=\tilde{Y}/G$, where $\tilde{Y}$ is a fibered $3$-manifold, and $G$ is a finite group 
acting on $\tilde{Y}$ which preserves a fibration $\tilde{Y}\rightarrow\s^1$ such that the induced action on $\s^1$ is orientation-preserving.
\end{conjecture}

\noindent{\bf Remarks:}\hspace{1mm}
(1) Seminal works of Taubes (concerning near symplectic structure \cite{Tau}) and Kronheimer 
(on minimal genus \cite{Kron}) led to the following conjecture: If $\s^1\times M^3$ is symplectic,
$M^3$ must be fibered over $\s^1$. First progress was made in \cite{CM} where the conjecture
was confirmed under a stronger assumption, i.e., $\s^1\times M^3$ admits a symplectic
Lefschetz fibration. After a series of partial results, the conjecture was finally resolved by Friedl
and Vidussi in \cite{FV}. The extension of the conjecture to circle bundles over $3$-manifold
was considered in \cite{Bald1} (the question was raised even earlier in \cite{FGM}), and partial
results were obtained in \cite{Bow} and \cite{FV1}. 

(2) Conjecture 1.7 is true if the fixed-point free circle action on $X$ preserves the symplectic 
structure. This follows easily from a generalized moment map argument, see \cite{McD, FGM}. 

(3) Conjecture 1.7 also holds true if $\text{rank }z(\pi_1(X))>1$. It is shown in \cite{C2} that
any Seifert fibration on such a $4$-manifold may be extended to a principal $T^2$-bundle over 
a $2$-orbifold, which can be given a compatible complex structure. If the complex structure 
is non-K\"{a}hler, then $X$ can not be symplectic unless it is a primary Kodaira surface (cf.
\cite{Biq}, also \cite{C1}), and in this case the circle action is free and $Y$ is fibered. If $X$ 
is K\"{a}hler, then $X=(T^2\times \Sigma)/G$ for some free finite group action of $G$ preserving 
the product structure on $T^2\times \Sigma$, cf. \cite{FM}, Theorem 7.7 in p. 200, in which case
$\tilde{Y}=\s^1\times \Sigma$ with the $G$-invariant fibration $pr_1: \s^1\times \Sigma\rightarrow
\s^1$.

\vspace{2mm}

Generalizing the work of Friedl-Vidussi \cite{FV1}, we obtain the following result.

\begin{proposition}
Let $X$ be a symplectic $4$-manifold and $\pi:X\rightarrow Y$ be a Seifert fibration. Then there
exists a $3$-manifold $\tilde{Y}$ and a finite group action of $G$ such that $Y=\tilde{Y}/G$.
Furthermore, if the canonical class $K_X$ is torsion, then $\tilde{Y}$ is fibered, and there is a 
$G$-invariant fibration $\tilde{Y}\rightarrow\s^1$ such that the induced $G$-action on the 
base $\s^1$ is orientation-preserving.
\end{proposition}

\noindent{\bf Additional Remarks:}\hspace{1mm}
It was shown in \cite{Bald3} that if a symplectic $4$-manifold admits a smooth circle action
with nonempty fixed-points, it must be diffeomorphic to a rational or ruled surface. Thus
the smooth classification of symplectic $4$-manifolds with smooth circular symmetry hinges on
the resolution of Conjecture 1.7, which remains open in the following case: $X$ has Kodaira
dimension $1$ (cf. \cite{Li}), and $\text{rank }z(\pi_1(X))=1$ (or equivalently, $X$ supports
no complex structures). 

\vspace{2mm}

\noindent{\bf Acknowledgments}: I wish to thank Ron Fintushel for communications regarding the
knot surgery formula for Seiberg-Witten invariants, Dieter Kotschick for bringing to my attention 
related works in the literature, and Inanc Baykur for interesting comments. This work was initiated during a visit to the Max Planck Institute for Mathematics, Bonn, in Summer 2010. The paper was completed in the current form during my sabbatical leave in Fall 2011 at the Institute for Advanced Study, which was supported by a grant from The S.S. Chern Fund and by NSF grant DMS-0635607. 
I wish to thank both institutes for their hospitality and financial support.

\section{New constraints of Seifert fibered $4$-manifolds}

\begin{lemma}
Suppose $X$ has nonzero Seiberg-Witten invariant. Let $\pi:X\rightarrow Y$ be any Seifert
fibration where $b_1(Y)>1$. Then the $3$-orbifold $Y$ has nonzero Seiberg-Witten invariant.
\end{lemma}

\begin{proof}
Let $\L$ be a $Spin^c$-structure such that $SW_X(\L)\neq 0$. If $b_2^{+}>1$, then $\L=\pi^\ast
\L_0$ for some $Spin^c$-structure $\L_0$ on $Y$, and moreover
$$
SW_X(\L)=\sum_{\L^\prime\equiv \L_0\mod{\chi }}SW_Y(\L^\prime),
$$
where $\chi$ stands for the Euler class of $\pi$, cf. Baldridge \cite{Bald2}, Theorem C. It follows
readily that $Y$ has nonzero Seierg-Witten invariant in this case.

When $b_2^{+}=1$ and $b_1(Y)>1$, the above formula continues to hold  (cf. \cite{Bald2}, 
Corollary 25) as long as $\L=\pi^\ast\L_0$. Thus it remains to show that $\L=\pi^\ast\L_0$ 
for some $Spin^c$-structure $\L_0$ on $Y$.

In order to see this, we observe first that $H^2(X;\Z)/Tor$ has rank $2$, and any torsion element of 
$H^2(X;\Z)$ is the $c_1$ of the pull-back of an orbifold complex line bundle on $Y$.
Moreover, there exists an embedded loop $\gamma$ lying in the complement of the singular set 
of $Y$, such that an element of $H^2(X;\Z)$ is the $c_1$ of the pull-back of an 
orbifold complex line bundle on $Y$ if and only if it is a multiple of the Poincar\'{e} dual of 
the $2$-torus $T\equiv \pi^{-1}(\gamma)\subset X$ in $H^2(X;\Z)/Tor$, see Baldridge \cite{Bald2}, Theorem 9.  With this understood, $\L=\pi^\ast \L_0$ for some $Spin^c$-structure $\L_0$ on $Y$ if and only if $c_1(\L)$ is Poincar\'{e} dual to a multiple of $T$ over $\Q$. 

We shall prove this using the product formula of Seiberg-Witten invariants in Taubes \cite{T2}. To 
this end, we let $N$ be the boundary of a regular neighborhood of $T$ in $X$. Then $N$ is a 
$3$-torus, essential in the sense of \cite{T2}, which splits $X$ into two pieces $X_{+}$, $X_{-}$. 
With this understood, Theorem 2.7 of \cite{T2} asserts that $SW_X(\L)$ can be expressed as a 
sum of products of the Seiberg-Witten invariants of $X_{+}$ and $X_{-}$. In particular, $c_1(\L)$ 
is expressed as a sum of $x$, $y$, where $x,y$ are in the images of $H^2(X_{+},N;\Z)$ and $H^2(X_{-},N;\Z)$ in $H^2(X;\Z)$ respectively. It is easily seen that mod torsion elements both are generated by the Poincar\'{e} dual of $T$. Hence the claim $\L=\pi^\ast \L_0$.

\end{proof}

\begin{corollary}
Let $\pi:X\rightarrow Y$ be a Seifert fibration where $b_1(Y)>1$. If the underlying $3$-manifold 
$|Y|$ contains a non-separating $2$-sphere, then $X$ has vanishing Seiberg-Witten invariant. 
\end{corollary}

\begin{proof}
If $X$ has nonzero Seiberg-Witten invariant, so does $Y$ by Lemma 2.1. By the 
de-singularization formula in \cite{C1}, the underlying $3$-manifold $|Y|$ also has nonzero
Seiberg-Witten invariant. But if $|Y|$ contains a non-separating $2$-sphere, we have a
decomposition $|Y|=Y_1\# \s^1\times\s^2$ where $b_1(Y_1)=b_1(Y)-1>0$. This implies 
that $|Y|$ has vanishing Seiberg-Witten invariant, which is a contradiction.

\end{proof}

Recall that a $3$-orbifold is called {\it pseudo-good} if it contains no bad $2$-suborbifold. 

\begin{lemma}
Let $X$ be a Seifert fibered $4$-manifold with Seifert fibration 
$\pi:X\rightarrow Y$. Then $Y$ is pseudo-good if either the Euler class of $\pi$ is torsion,
or $X$ has nontrivial Seiberg-Witten invariant.
\end{lemma}

\begin{proof}
We shall show that if $Y$ has a bad $2$-suborbifold $\Sigma$, then the Euler class of 
$\pi$ must be non-torsion and $X$ has vanishing Seiberg-Witten invariant. By definition, 
$\Sigma$ is an embedded $2$-sphere in the underlying $3$-manifold $|Y|$, such that either 
$\Sigma$ contains exactly one singular point of $Y$ or $\Sigma$ contains two singular points of different multiplicities. For simplicity, we let $p_1,p_2$ be the singular points on $\Sigma$, with 
$p_1, p_2$ contained in the components $\gamma_1, \gamma_2$ of the singular set of $Y$, with multiplicities $\alpha_1,\alpha_2$ respectively. We assume that $\alpha_1<\alpha_2$, 
with $\alpha_1=1$ representing the case where $\Sigma$ contains only one singular point 
$p_2$. Note that since $\alpha_1\neq \alpha_2$, $\gamma_1, \gamma_2$ are distinct.
 
We first show that the Euler class of $\pi$ must be non-torsion. Let $e(\pi)$ denote the image of
the Euler class of $\pi$ in $H^2(|Y|;\Q)$.
Then 
$$
\int_{|\Sigma|} e(\pi)=b+\beta_1/\alpha_1+\beta_2/\alpha_2, 
$$
where $b\in\Z, 1\leq \beta_1<\alpha_1$ or $\alpha_1=\beta_1=1$,  and 
$1\leq \beta_2<\alpha_2$ with $\text{gcd }(\alpha_2,\beta_2)=1$. 
It follows from the fact that $\alpha_1<\alpha_2$ and 
$\alpha_2, \beta_2$ are co-prime that $\int_{|\Sigma|} e(\pi)\neq 0$.
Hence the Euler class of $\pi$ is non-torsion. 

Now note that $b_1(Y)=b_2^{+}+1>1$ since the Euler class of $\pi$ is non-torsion (cf. \cite{Bald2}).
On the other hand, clearly $|\Sigma|$ is a non-separating $2$-sphere in $|Y|$. By Corollary 2.2,
$X$ has vanishing Seiberg-Witten invariant.

\end{proof}

\noindent{\bf Proof of Theorem 1.1}

\vspace{1mm}

We first introduce some notations. 
Let $y_0\in Y$ be a regular point and $F$ be the fiber of the Seifert fibration
over $y_0$, and fix a point $x_0\in F$. Denote by $i: F\hookrightarrow X$ the inclusion. 
We shall prove that $i_\ast: \pi_1(F,x_0)\rightarrow \pi_1(X,x_0)$ is injective first.

By Lemma 2.3, $Y$ is pseudo-good. As a corollary of the proven Thurston's Geometrization 
Conjecture for $3$-orbifolds (cf. \cite{BLP,MM}), $Y$ is very good, i.e, there is a $3$-manifold 
$\tilde{Y}$ with a finite group action $G$ such that $Y=\tilde{Y}/G$. Let $pr:\tilde{Y}\rightarrow Y$ 
be the quotient map, and let $\tilde{X}$ be the $4$-manifold which is the total space of the 
pull-back circle bundle of $\pi:X\rightarrow Y$ via $pr$.  Then $\tilde{X}$ has a natural free 
$G$-action such that $X=\tilde{X}/G$. We fix a point $\tilde{y}_0\in\tilde{Y}$ in the pre-image 
of $y_0$, and let $\tilde{F}\subset \tilde{X}$ be the fiber over $\tilde{y}_0$. We fix a 
$\tilde{x}_0 \in \tilde{F}$ which is sent to $x_0$ under $\tilde{X}\rightarrow X$.
Then consider the following commutative diagram 
$$
\begin{array}{llllllll}
\rightarrow & \pi_2(\tilde{Y},\tilde{y}_0) & \stackrel{\tilde{\delta}}{\rightarrow} & \pi_1(\tilde{F}, \tilde{x}_0) &
\stackrel{\tilde{i}_\ast}{\rightarrow} & \pi_1(\tilde{X},\tilde{x}_0) & \stackrel{\tilde{\pi}_\ast}{\rightarrow} & 
\pi_1(\tilde{Y},\tilde{y}_0) \rightarrow \\
                  &  \;\;\;\;\;\; \downarrow                    &                                              &   \;\;\;\;\;\;\parallel & 
                   & \;\;\;\;\; \;\downarrow &  & \;\;\;\;\;\;\downarrow \\
\rightarrow & \pi_2^{orb}(Y,y_0) & \stackrel{\delta}{\rightarrow} & \pi_1(F,x_0) &
\stackrel{i_\ast}{\rightarrow} & \pi_1(X,x_0) & \stackrel{\pi_\ast}{\rightarrow} & 
\pi_1^{orb}(Y,y_0) \rightarrow .\\                  
\end{array}
$$
Since $\pi_1(\tilde{X},\tilde{x}_0)\rightarrow \pi_1(X,x_0)$ is injective, it follows that $i_\ast$ is
injective if $\tilde{\delta}$ has zero image.

With the preceding understood, by the Equivariant Sphere Theorem of Meeks and Yau 
(cf. \cite{MY}, p. 480), there are disjoint, embedded $2$-spheres $\tilde{\Sigma}_i$ in $\tilde{Y}$ 
such that the union of $\tilde{\Sigma}_i$ is invariant under the $G$-action and the classes 
of $\tilde{\Sigma}_i$ generate $\pi_2(\tilde{Y})$ as a $\pi_1(\tilde{Y})$-module. Suppose the 
image of $\tilde{\delta}$ is non-zero. Then there is a $\tilde{\Sigma}\in\{\tilde{\Sigma}_i\}$ which
is not in the kernel of $\tilde{\delta}$. It follows that the Euler class of 
$\tilde{\pi}:\tilde{X}\rightarrow \tilde{Y}$ is nonzero on $\tilde{\Sigma}$, and consequently,
$\tilde{\Sigma}$ is a non-separating $2$-sphere in $\tilde{Y}$. It also follows that the Euler class
of $\pi$ is non-torsion, and in this case, we have $b_1(Y)=b_2^{+}+1>1$.

We first consider the case where there are no elements of $G$ which leaves the $2$-sphere 
$\tilde{\Sigma}$ invariant. In this case the image of $\tilde{\Sigma}$ in $Y$ under the quotient 
$\tilde{Y}\rightarrow Y$ is an embedded non-separating $2$-sphere which contains no singular points of $Y$.  By Corollary 2.2, $X$ has vanishing Seiberg-Witten invariant,
a contradiction.

Now suppose that there are nontrivial elements of $G$ which leave $\tilde{\Sigma}$ invariant. We denote by $G_0$ the maximal subgroup of $G$ which leaves $\tilde{\Sigma}$ invariant. We shall 
first argue that the action of $G_0$ on $\tilde{\Sigma}$ is orientation-preserving. Suppose not, and 
let $\tau\in G_0$ be an involution which acts on $\tilde{\Sigma}$ reversing the orientation. Then 
since $\tau$ preserves the Euler class of $\tilde{\pi}$ and reverses the orientation of $\tilde{\Sigma}$, the Euler class of $\tilde{\pi}$ evaluates to $0$ on $\tilde{\Sigma}$, which is a contradiction. Hence 
the action of $G_0$ on $\tilde{\Sigma}$ is orientation-preserving.

Let $\Sigma$ be the image of $\tilde{\Sigma}$ in $Y$. Then $\Sigma$ is a non-separating, 
spherical $2$-suborbifold of $Y$. Consequently $|Y|$ contains a non-separating $2$-sphere,
which, by Corollary 2.2, implies that $X$ has vanishing Seiberg-Witten invariant, a contradiction.

For part (2) of the theorem, we claim that there exist embedded $2$-spheres $C_1,\cdots,C_N$ of 
self-intersection $0$ which generate the image of $\pi_2(X)\rightarrow H_2(X)$. Assume the 
claim momentarily. Since $X$ has $b_2^{+}>1$ and nonzero Seiberg-Witten invariant, a theorem 
of Fintushel and Stern (cf. \cite{FS}) asserts that each $C_i$ must be torsion in $H_2(X)$, from
which part (2) follows.

To see the existence of $C_1,\cdots,C_N$, note that the restriction of $\tilde{\pi}$ over each
$\tilde{\Sigma}_i$ must be trivial from the proof of part (1). We claim that there are sections 
$\Sigma_i$ of $\tilde{\pi}$ over $\tilde{\Sigma}_i$ or a push-off of it, such that no nontrivial 
element of $G$ will leave $\Sigma_i$ invariant. The image of such a $\Sigma_i$ under 
$\tilde{X}\rightarrow X$ is an embedded $2$-sphere of self-intersection $0$,  which will be 
our $C_i$. Since $\tilde{\pi}_\ast: \pi_2(\tilde{X})\rightarrow \pi_2(\tilde{Y})$ is isomorphic and 
$\pi_2(\tilde{X})=\pi_2(X)$, the $2$-spheres $C_i$ generate the image of 
$\pi_2(X)\rightarrow H_2(X)$.

To construct these $\Sigma_i$'s, note that there are three possibilities: (i) no nontrivial element
of $G$ leaves $\Sigma_i$ invariant, (ii) there is a nontrivial $g\in G$ (and no other elements) which
acts on $\Sigma_i$ preserving the orientation, and (iii) there is an involution $\tau\in G$ (and no
other elements) which acts on $\Sigma_i$ reversing the orientation. The existence of $\Sigma_i$
is trivial in case (i). In case (ii), any section of $\tilde{\pi}$ will do because $g$ acts as translations
on the fiber of $\tilde{\pi}$. In case (iii), we take $\Sigma_i$ to be a section of $\tilde{\pi}$ over
a push-off of $\tilde{\Sigma}_i$.  This shows the existence of $\Sigma_i$'s, and the proof of 
Theorem 1.1 is completed.

\vspace{1mm}

\noindent{\bf Proof of Theorem 1.2}

\vspace{1mm}

According to the Enriques-Kodaira classification, $X$ falls into one of the following
classes of surfaces: class $VII$, rational or ruled, elliptic, $K3$, or general type. 

First consider the case where $X$ is a class $VII$ surface. 
Since $b_1=1$, the Euler number of $X$ equals $b_2$. 
This shows that if $X$ is Seifert fibered, $b_2$ must be zero.

Suppose $X$ is of general type or $K3$. Since $c_2(X)>0$ (cf. \cite{BPV}, Theorem 1.1 
in Chapter VII for the case of general type), $X$ can not be Seifert fibered. 

Suppose $X$ is elliptic. Then the Euler number of $X$ is non-negative. 
Hence if $X$ is Seifert fibered, $X$ must be a minimal elliptic surface with Euler number $0$. 
Suppose the monodromy representation of the elliptic fibration is nontrivial. Then $X$ is K\"{a}hler 
(cf. \cite{FM}, Theorem 7.8, p.201), and therefore $z(\pi_1(X))$ is infinite by Theorem 1.1. 
On the other hand, since the monodromy is nontrivial, the image of the $\pi_1$ of a generic fiber 
of the elliptic fibration does not lie in the center $z(\pi_1(X))$, which implies that the base of the 
elliptic fibration must be a nonsingular torus. In this case $X$ is a bi-elliptic surface, which admits another elliptic fibration with trivial monodromy representation.

Finally, let $X$ be rational or ruled. If $X$ is Seifert fibered, then the vanishing of Euler number 
and signature implies that $X$ must be a ruled surface over an elliptic curve. In particular, 
$\pi_1(X)=\Z\times \Z$ and $\pi_2(X)\neq 0$. By the classification result concerning such Seifert fibered $4$-manifolds in \cite{C2}, $X$ is diffeomorphic to $T^2\times \s^2$.

This completes the proof of Theorem 1.2.

\vspace{1mm}

\noindent{\bf Proof of Theorem 1.3}

\vspace{1mm}

Since $T^4$ has $b_2^{+}>1$ and nonzero Seiberg-Witten invariant, any smooth circle action
on $T^4$ must be fixed-point free by Baldridge \cite{Bald3}. Let $\pi: T^4\rightarrow Y$ be
the corresponding Seifert fibration. By Lemma 2.3, $Y$ is pseudo-good, hence there is a 
$3$-manifold $\tilde{Y}$ and a finite group action of $G$ such that $Y=\tilde{Y}/G$. Since
$\pi_1(T^4)=\Z^4$, we see that the center of $\pi_1(\tilde{Y})$ must have rank $3$, and 
consequently, $\tilde{Y}=T^3$. Finally, observe that in this case the $\pi_1^{orb} $ of 
$Y=\tilde{Y}/G=T^3/G$ must have a torsionless center, and $Y=T^3$ must be true. 
It follows that $\pi: T^4\rightarrow Y=T^3$ is a trivial $S^1$-bundle, and the original circle
action is smoothly conjugate to a linear action. This proves Theorem 1.3.

\vspace{1mm}

\noindent{\bf Proof of Proposition 1.8}

\vspace{1mm}

First, every symplectic $4$-manifold has nonzero Seiberg-Witten invariant by work of Taubes
\cite{T1}. More precisely, the Seiberg-Witten invariant associated to the canonical $Spin^c$-structure
$\I\oplus K^{-1}_X$ equals $\pm 1$, where in the case of $b_2^{+}=1$, the convention is to orient
$H^{2,+}(X;\R)$ by the symplectic form. Hence by Lemma 2.3, $Y$ is pseudo-good, and 
there exists a $3$-manifold $\tilde{Y}$ and a finite group action of $G$ such that $Y=\tilde{Y}/G$
(cf. \cite{BLP,MM}). We denote by $\tilde{\pi}:\tilde{X}\rightarrow \tilde{Y}$ the pull-back of $\pi$
via $\tilde{Y}\rightarrow Y$. Then $pr:\tilde{X}\rightarrow X$ is a finite cover of $X$, hence 
$\tilde{X}$ is also symplectic. 

Now suppose $K_X$ is torsion. Then $K_{\tilde{X}}=pr^\ast K_X$ is also torsion, and by 
\cite{FV1}, $\tilde{Y}$ must be a $T^2$-bundle over $\s^1$. 
We will show that there exists a $G$-invariant fibration $\tilde{Y}\rightarrow\s^1$ such 
that the induced action of $G$ on the base $\s^1$ is orientation-preserving. The discussion 
will be according to the following three cases.

{\bf Case 1}: $b_1(\tilde{Y})=1$. In this case, our claim follows from Edmonds-Livingston \cite{EL}, Theorem 5.2, together with Meeks-Scott \cite{MS}, Theorem 8.1, since $b_1(Y)$, which
is nonzero, must also be equal to $1$, and the condition $H_1(\tilde{Y};\Q)^G=\Q$ in
\cite{EL} is verified. 

{\bf Case 2}: $b_1(\tilde{Y})=2$. In this case, $\tilde{Y}$ is a nontrivial $\s^1$-bundle over $T^2$.  Note that $z(\pi_1(\tilde{Y}))$ has rank $1$, so it must be preserved under $G$. 
By Meeks-Scott \cite{MS}, the $G$-action is equivalent to a fiber-preserving $G$-action. 
If the induced action on the base $T^2$ is homologically nontrivial, $Y$ would have $b_1=0$, 
which is a contradiction. Hence the induced action on $T^2$ is homologically trivial, which must be given by translations. In this case, it is easily seen that the $T^2$-bundle over $\s^1$ structure on 
$\tilde{Y}$ is $G$-invariant, with the induced action preserving the orientation of the base $\s^1$.

{\bf Case 3}: $b_1(\tilde{Y})=3$. In this case, $\tilde{Y}=T^3$. By Meeks-Scott \cite{MS}, 
the $G$-action is smoothly conjugate to a linear action. Furthermore, since $Y=\tilde{Y}/G$ has 
only $1$-dimensional singular set, the $G$-action must preserve a product structure 
$T^3=\s^1\times T^2$. This case follows easily. (Note that there is a finite group action on $T^3$ which does not preserve any product structure $T^3=\s^1\times T^2$, cf. Meeks-Scott \cite{MS}, 
p. 291. Of course, such an action has an isolated fixed point.)

\section{Destroying smooth circle actions via knot surgery}

We begin by reviewing the knot surgery and the knot surgery formula for Seiberg-Witten invariants,
following Fintushel-Stern \cite{FinS}. To this end, let $M$ be a smooth $4$-manifold with 
$b_2^{+}>1$,
which possesses an essential embedded torus $T$ of self-intersection $0$. Given any knot 
$K\subset\s^3$, the knot surgery of $M$ along $T$ with knot $K$ is the smooth $4$-manifold, 
denoted by $M_K$, which is constructed as follows:
$$
M_K\equiv M\setminus Nd(T)\cup_\phi (\s^3\setminus Nd(K))\times \s^1,
$$ 
where the only requirement on the identification $\phi$ is that it sends the meridian of $T$ to the 
longitude of $K$. It follows easily from the Mayer-Vietoris sequence that the integral homology
of $M_K$ is naturally identified with the integral homology of $M$, under which the intersection pairings on $M$ and $M_K$ agree. (In \cite{FinS}, it is assumed that $M$ is simply connected, 
$T$ is $c$-embedded, and $\pi_1(M\setminus T)$ is trivial. These assumptions are irrelevant to 
the discussions here.)

An important aspect of knot surgery is that the Seiberg-Witten invariant of $M_K$ can be computed
from that of $M$ and the Alexander polynomial of $K$, through the so-called knot surgery formula.
In order to state the formula, we let 
$$
\underline{SW}_M=\sum_z (\sum_{c_1(\L)=z} SW_M (\L)) t_z, \;\;\; z\in H^2(M;\Z),
$$
where $t_z\equiv \exp(z)$ is a formal variable, regarded as an element of the group ring 
$\Z H^2(M;\Z)$.

\begin{theorem}
{\em(}Knot Surgery Formula \cite{FinS}{\em)}
With the integral homology of $M$ and $M_K$ naturally identified, one has
$$
\underline{SW}_{M_K}=\underline{SW}_M \cdot \Delta_K(t),
$$
where $t=\exp(2[T])$. Here $\Delta_K(t)$ is the Alexander polynomial of $K$, and $[T]$ stands
for the Poincar\'{e} dual of the $2$-torus $T$. 
\end{theorem}

\noindent{\bf Remarks:}\hspace{1mm}
(1) The knot surgery formula was originally proved in \cite{FinS} under the assumption that 
$T$ is $c$-embedded. However, the assumption of $c$-embeddedness of $T$ is not essential
in the argument and it may be removed (cf. Fintushel \cite{F3}). 

(2) We note that in Theorem 1.5, the condition $SW_X(z)\equiv \sum_{c_1(\L)=z} SW_X(\L)\neq 0$
for some $z\in H^2(X;\Z)$ is equivalent to $\underline{SW}_X\neq 0$, which implies, by 
Theorem 3.1, that $\underline{SW}_{X_K}\neq 0$ for any knot $K$. Consequently, $X_K$ has nonzero Seiberg-Witten invariant. Moreover, if $H^2(X;\Z)$ has no $2$-torsions, so does 
$H^2(X_K;\Z)$, and in this case, $X$ and $X_K$ have the same Seiberg-Witten invariant 
provided that the Alexander polynomial of $K$ is trivial.

\vspace{2mm}

Now we give a proof for Theorem 1.5. To this end, we introduce the following
notations. Let $Y_1\equiv Y\setminus Nd(\pi(l))$ and 
$X_1\equiv X\setminus Nd(T)$, where $T$ is the $2$-torus $T=\pi^{-1}(\pi(l))$, such
that $X_1$ is the restriction of $\pi$ to the $3$-orbifold $Y_1$. 
Note that the boundary $\partial X_1$ is a $3$-torus $T^3$. We fix three embedded loops in 
$\partial X_1$: $m$, a meridian of $T$, $h$, a fiber of $\pi$, and $l^\prime$, a push-off of $l$, 
which all together generate $\pi_1(\partial X_1)$. We assume $m$, $l^\prime$ are sections over 
$\pi(m)$, $\pi(l^\prime)$ respectively. Finally, recall that the underlying manifolds of $Y_1$, $Y$ are
denoted by $|Y_1|$, $|Y|$ respectively. 

\begin{lemma}
The map $i_\ast:\pi_1(\partial X_1)\rightarrow \pi_1(X_1)$ induced by the inclusion is injective.
\end{lemma}

\begin{proof}
Suppose to the contrary that $i_\ast:\pi_1(\partial X_1)\rightarrow \pi_1(X_1)$ has a nontrivial 
kernel. Then $\pi_1(\partial Y_1)\rightarrow \pi_1(|Y_1|)$ must also have a nontrivial kernel. To see this, suppose $\gamma\neq 0$ lies in the kernel of $i_\ast$. Then $\pi_\ast (\gamma)$ lies in the kernel of $\pi_1(\partial Y_1)\rightarrow \pi_1^{orb}(Y_1)$. On the other hand, if 
$\pi_\ast (\gamma)=0$, then $\gamma$ is a multiple of $[h]$, which is zero in $\pi_1(X_1)$
by assumption. This contradicts to the fact that $[h]$ has infinite order in $\pi_1(X)$ 
(cf. Theorem 1.1).  Hence $\pi_1(\partial Y_1)\rightarrow \pi_1^{orb}(Y_1)\rightarrow \pi_1(|Y_1|)$ 
has a nontrivial kernel. 

Now by the Loop Theorem, $\partial Y_1$ is compressible in $|Y_1|$. This means that there is an embedded disc $D\subset |Y_1|$ such that (1) $D\cap \partial Y_1=\partial D$, (2) $\partial D$ is a homotopically nontrivial simple closed loop in $\partial Y_1$. We claim that $\partial D$ must be a copy of the meridian $\pi(m)$ of $\pi(l)$ in $Y$. 

To see this, write $[\partial D]=s \cdot [\pi(m)] +t \cdot [\pi(l^\prime)]$ in $\pi_1(\partial Y_1)$.
If $t\neq 0$, then $[\pi(l)]$ has finite order in $\pi_1^{orb}(Y)$. This implies that a nontrivial 
power of $[l]$ lies in the subgroup generated by $[h]$, which lies in $z(\pi_1(X))$, a contradiction 
to (ii). Hence $t=0$, and $s=1$ with $[\partial D]=[\pi(m)]$, so that $\partial D$ is isotopic to
$\pi(m)$. 

With the preceding understood, the median $\pi(m)$ bounds an embedded disk 
$D\subset |Y_1|$. It follows that there is a non-separating $2$-sphere in $|Y|$ intersecting 
with $\pi(l)$ in exactly one point. By Corollary 2.2, $X$ has vanishing Seiberg-Witten invariant, 
a contradiction. 

\end{proof}

\noindent{\bf Proof of Theorem 1.5}

\vspace{2mm}

Since $K$ is a nontrivial knot, $\pi_1(\partial((\s^3\setminus Nd(K))\times \s^1))\rightarrow \pi_1((\s^3\setminus Nd(K))\times \s^1)$ is injective. With Lemma 3.2, this implies that 
$\pi_1(X_K)$ is an amalgamated free product 
$$
\pi_1(X_K)= \pi_1(X_1) \ast _{\pi_1(T^3)} \pi_1((\s^3\setminus Nd(K))\times \s^1),
$$
where $\pi_1(T^3)=\pi_1(\partial X_1)=\pi_1(\partial ((\s^3\setminus Nd(K))\times \s^1))$. 

Since $K$ is not a torus knot, the knot group of $K$ has trivial center (cf. \cite{Rolf}). This implies
that the center of $\pi_1((\s^3\setminus Nd(K))\times \s^1)$ is generated by the class of 
$\{pt\}\times \s^1$. On the other hand, by the construction of $X_K$, $\{pt\}\times \s^1$ is identified 
with a push-off of $l$, and furthermore, by condition (ii) on $l$, no nontrivial powers of the homotopy class of $l$ is contained in the center $z(\pi_1(X))$. Consequently, $z(\pi_1(X_K))$ is trivial. By 
Theorem 1.1 (together with Baldridge \cite{Bald3}, Theorem 1.1), $X_K$ does not support any 
smooth circle actions.

\vspace{2mm}

{\Small University of Massachusetts, Amherst.\\
{\it E-mail:} wchen@math.umass.edu

\end{document}